\documentclass[11pt,reqno]{amsart}

\usepackage{color}
\definecolor{red}{rgb}{1,0,0}
\definecolor{green}{rgb}{0,1,0}
\definecolor{blue}{rgb}{0,0,1}
\definecolor{refkey}{gray}{.625}
\definecolor{labelkey}{gray}{.625}

\let\oldmarginpar\marginpar
\renewcommand\marginpar[1]{\-\oldmarginpar[\raggedleft\footnotesize #1]%
{\raggedright\footnotesize #1}}

\usepackage{verbatim}

\usepackage[T1]{fontenc}
\usepackage{lmodern}
\usepackage[latin1]{inputenc}
\usepackage{setspace}
\usepackage{indentfirst}
\usepackage{geometry}
\usepackage{hyperref}

\usepackage[nobysame,msc-links]{amsrefs}
\newcommand{\arXiv}[1]{\href{http://arxiv.org/abs/#1}{\texttt{arXiv:#1}}}
\renewcommand{\eprint}[1]{\arXiv{#1}}

\usepackage{amsmath}
\usepackage{amssymb}
\usepackage{mathrsfs}
\usepackage{stmaryrd}
\usepackage{mathtools}

\usepackage{graphicx}
\usepackage[all]{xy}
\def\dar[#1]{\ar@<2pt>[#1]\ar@<-2pt>[#1]} 

\usepackage{amsthm}

\theoremstyle{plain}
\newtheorem{prop}{Proposition}[section]
\newtheorem{lem}[prop]{Lemma}

\newtheorem{cor}[prop]{Corollary}
\newtheorem{thm}[prop]{Theorem}

\newtheorem*{prop*}{Proposition}
\newtheorem*{lem*}{Lemma}
\newtheorem*{sublem*}{Sublemma}
\newtheorem*{cor*}{Corollaire}
\newtheorem*{thm*}{Theorem}
\newtheorem*{hypo*}{Hypothesis}
\newtheorem*{question*}{Question}
\newtheorem*{conjecture*}{Conjecture}
\newtheorem*{scholum*}{Scholum}
\newtheorem{defn}[prop]{Definition}
\newtheorem*{defn*}{Definition}
\newtheorem*{fact*}{Fact}

\newtheorem{rmk}[prop]{Remark}

\newtheorem*{rmk*}{Remark}
\newtheorem*{rmks*}{Remarks}

\newtheoremstyle{slanted}
  {3pt}
  {3pt}
  {\slshape}
  {}
  {\bfseries}
  {.}
  {.5em}
  {}

\theoremstyle{slanted}
\newtheorem{ex}[prop]{Example}

\newtheorem*{ex*}{Example}
\newtheorem*{exs*}{Examples}

\newtheorem*{notation*}{Notation}

\theoremstyle{definition}

\newtheorem*{con*}{Construction}
\newtheorem*{note*}{Note}

\theoremstyle{remark}

\newtheorem*{warning*}{Warning}
\newtheorem*{shortnote*}{Note}
\newtheorem*{claim*}{Claim}
\newtheorem*{axiom*}{Axiom}

\DeclareMathOperator{\sgn}{sgn}

\DeclareMathOperator{\iso}{\tilde{\Theta}}

\newcommand{\RR}{\mathbb{R}}

\newcommand{\ZZ}{\mathbb{Z}}

\newcommand{\thalf}{\tfrac{1}{2}}

\newcommand{\rond}{\circ}

\newcommand{\cinf}[1]{C^{\infty}(#1)}
\newcommand{\sections}[1]{\Gamma(#1)}
\newcommand{\XX}{\mathfrak{X}} 
\newcommand{\OO}{\Omega} 

\newcommand{\xto}[1]{\xrightarrow{#1}}

\newcommand{\injection}{\hookrightarrow}
\newcommand{\surjection}{\twoheadrightarrow}

\newcommand{\ortho}{^{\perp}} 
\newcommand{\inv}{^{-1}}

\newcommand{\mfg}{\mathfrak{g}}

\newcommand{\DD}{\mathcal{D}}

\newcommand{\ld}[1]{L_{#1}}

\newcommand{\duality}[2]{\langle #1| #2\rangle}
\newcommand{\ip}[2]{(#1,#2)}
\newcommand{\cb}[2]{\llbracket #1,#2\rrbracket}
\newcommand{\lb}[2]{[#1,#2]}
\newcommand{\db}[2]{#1\circ #2}
\newcommand{\anchor}{a}

\newcommand{\dee}{d}

\newcommand{\dsmile}{\breve{\mathrm d}}

\newcommand{\LieDer}{L}

\newcommand{\inserts}{\iota}
\newcommand{\qm}{-}

\newcommand{\set}[1]{\left\{#1\right\}}
\newcommand{\LieG}{\mfg}

\newcommand{\defbe}{:=}

\newcommand{\Connection}{\kappa}

\newcommand{\Curvature}{R}
\newcommand{\ConnectionDer}{\nabla}
\newcommand{\DiffConnection}{J}
\newcommand{\JJJ}{J}
\newcommand{\algebroidlb}[2]{[#1,#2]_{\Abd}}
\newcommand{\lbG}[2]{[#1,#2]^\mathcal{G}}
\newcommand{\ipG}[2]{(#1,#2)^\mathcal{G}}
\newcommand{\lbLieG}[2]{[#1,#2]^{\LieG}}
\newcommand{\ipLieG}[2]{(#1,#2)}

\newcommand{\Abd}{\mathcal{A}}
\newcommand{\Abds}{\mathcal{A}^*}
\newcommand{\PontryaginCocycle}{\langle\Curvature^{\Connection}\wedge\Curvature^{\Connection}\rangle}
\newcommand{\PontryaginCocycleR}{\langle\Curvature\wedge\Curvature\rangle}
\newcommand{\Fs}{F^*}
\newcommand{\HForm}{\mathcal{H}}
\newcommand{\PForm}{\mathcal{P}}

\newcommand{\dF}{d^F}
\newcommand{\varphis}{\varphi^*}
\newcommand{\CP}{\text{c.p.}}
\newcommand{\tauinverse}{\tau^{-1}}
\newcommand{\inverse}{^{-1}}
\newcommand{\Econnection}{\nabla^E}
\newcommand{\Fconnection}{\nabla^F}

\newcommand{\splitting}{\lambda}
\newcommand{\ProjectionTo}[1]{\mathrm{Pr}_{#1}}

\newcommand{\GG}{\mathcal{G}}
\newcommand{\good}{coherent}

\newcommand{\anchorUps}{\anchor^*}

\newcommand{\class}[1]{ [{#1}]}

\newcommand{\Bijectivemap}{ \mathfrak{b} }
\newcommand{\Bijectivemapinverse}{\mathfrak{i}}
\newcommand{\ICRC}{\mathfrak{R}}
\newcommand{\ECCC}{\mathfrak{C}}
\newcommand{\Chain}{\mathbf{C}}

\newcommand{\ample}{\mathcal{A}}
\newcommand{\ipg}[2]{\ip{#1}{#2}^{\GG}}
\newcommand{\gq}{\boldsymbol{q}}
\newcommand{\gr}{\boldsymbol{r}}
\newcommand{\gs}{\boldsymbol{s}}
\newcommand{\gt}{\boldsymbol{t}}
\newcommand{\ax}{x}
\newcommand{\fx}{x}
\newcommand{\fy}{y}
\newcommand{\fz}{z}
\newcommand{\dissection}{\Psi}
\newcommand{\dample}{\check{d}}
\newcommand{\hoist}{\kappa}
\newcommand{\hoistcurvature}{R^\kappa}
\newcommand{\subcomplex}{\Chain_{\leftrightarrow}}
\newcommand{\fsgf}{F^*\oplus\GG\oplus F}
\newcommand{\ISO}{\Theta}

\hypersetup{pdfpagemode=UseOutlines,colorlinks=false,pdfpagelayout=SinglePage,pdfstartview=FitH,bookmarksopen=true}

\begin{document}

\title{On Regular Courant Algebroids}
\author[Chen]{Zhuo Chen}
\thanks{Research partially supported by NSFC grant 10871007.}
\address{Tsinghua University, Department of Mathematics}
\email{\href{mailto:zchen@math.tsinghua.edu.cn}{\texttt{zchen@math.tsinghua.edu.cn}}}
\author[Sti\'enon]{Mathieu Sti\'enon}
\address{Universit\'e Paris Diderot\footnote{(Paris 7)}, Institut de Math\'ematiques de Jussieu\footnote{(CNRS, UMR 7586)}}
\email{\href{mailto:stienon@math.jussieu.fr}{\texttt{stienon@math.jussieu.fr}}}
\author[Xu]{Ping Xu}
\thanks{Research partially supported by NSF grants DMS-0605725 and DMS-0801129.}
\address{Pennsylvania State University, Department of Mathematics}
\email{\href{mailto:ping@math.psu.edu}{\texttt{ping@math.psu.edu}}}

\begin{abstract}
For any regular Courant algebroid, we construct a characteristic class
\textit{\`a la} Chern-Weil. This intrinsic invariant of the Courant algebroid
is a degree-$3$ class in its naive cohomology.
When the Courant algebroid is exact, it reduces to the \v{S}evera class (in $H^3_{DR}(M)$).
On the other hand, when the Courant algebroid is a quadratic Lie algebra $\mfg$,
it coincides with the class of the Cartan $3$-form (in $H^3(\mfg)$).
We also give a complete classification of regular Courant algebroids
and discuss its relation to the characteristic class.
\end{abstract}

\maketitle

\tableofcontents

\section*{Introduction}

Courant algebroids were introduced in~\cite{MR1472888} as a way to
merge the concept of Lie bialgebra and the bracket on
$\XX(M)\oplus\OO^1(M)$ first discovered by Courant~\cite{MR998124}
--- here $M$ is a smooth manifold. Roytenberg gave an equivalent
definition phrased in terms of the Dorfman
bracket~\cite{math/9910078}, which highlighted the relation of
Courant algebroids to $L_{\infty}$-algebras~\cite{MR1183483}
observed by Roytenberg \& Weinstein~\cite{MR1656228}.

Despite its importance, the subject suffered from the lack of
examples for a long time. One important class of Courant algebroids,
called exact Courant algebroids, was discovered by
\v{S}evera~\cite{Severa}. A Courant algebroid $E$ is said to be
exact if its underlying vector bundle fits into an exact sequence
$0\to T^*M\xto{\rho^*} E\xto{\rho}TM\to  0$, where $\rho$ is the
anchor map. \v{S}evera proved that the isomorphism classes of exact
Courant algebroids are classified by a degree-$3$ class in the de
Rham cohomology of $M$: the \v{S}evera class. Furthermore, the
structure of  transitive Courant algebroids --- Courant algebroids
with surjective anchors --- was described independently by
Vaisman~\cite{MR2178250}, \v{S}evera (in a private correspondence
with Weinstein~\cite{Severa}) and Bressler~\cite{MR2360313}.
\v{S}evera and Bressler also classified transitive Courant
algebroids as extensions of transitive Lie
algebroids~\cites{MR2360313,Severa}. Indeed \v{S}evera also
 outlined some nice ideas of classification of transitive Courant
algebroids in \cite{Severa}.

In \cite{MathieuXu}, the second and third authors   introduced
the  modular class of Courant algebroids,
a degree 1 characteristic class in the naive cohomology
$H^1_{naive}(E)$ of the Courant algebroid $E$.
It is natural to ask whether there is a
degree 3 characteristic  class for a general Courant algebroid resembling
the Severa class, and if so, what role is played by such a class
in the classification of Courant algebroids.
The main purpose of this paper is to answer these questions for regular Courant
algebroids, that is, Courant algebroids with a constant rank anchor.
Such Courant algebroids $E$ are particularly easier to handle,
for their associated characteristic distributions $\rho(E)$
do not have any singularity.
  Note that
characteristic classes of Lie algebroids were studied
by Evens-Lu-Weinstein \cite{ELW}, Crainic and
Fernandes \cite{CF1, Fernandes}. It would be
interesting to explore if there is any intrinsic connection
between these constructions.


To any regular Courant algebroid, we will construct a degree 3
class, called the characteristic class, in the naive cohomology
$H^3_{naive}(E)$ of the Courant algebroid, an analogue of \v{S}evera's
class. Note that when $E$ is an exact Courant algebroid,
$H^3_{naive}(E)$ is isomorphic to the de Rham cohomology
$H^3_{DR}(M)$. However, unlike the exact case, $H^3_{naive}(E)$ does
not classify regular Courant algebroids. The classification problem
is much subtler in this case.

Given a regular Courant algebroid $E$ with anchor $\rho$,
$\ample:=E/(\ker\rho)\ortho$ is clearly a regular Lie algebroid,
i.e.\ a Lie algebroid with a constant rank anchor.  It
is called the ample Lie algebroid of $E$.
The kernel $\GG$ of the anchor of $\ample$
 is a bundle of quadratic Lie algebras,
which satisfies certain compatibility conditions.
Therefore, the Lie algebroid $\ample$ is quadratic (see Definition~\ref{def:quad}).

The first natural question is whether every
quadratic Lie algebroid $\ample$ arises in this way.
It turns out that there is an obstruction: the first Pontryagin class,
an element in $H^4(F)$ naturally associated to any quadratic Lie algebroid.
Here $F$ is the image of the anchor (an integrable subbundle of $TM$)
and $H^4(F)$ stands for the leafwise de Rham cohomology of $F$.
This obstruction is similar to the one described by \v{S}evera and Bressler
in the transitive case~\cite{MR2360313,Severa}.

To recover a Courant algebroid from a given a quadratic Lie algebroid $\ample$
with vanishing first Pontryagin class, one needs an extra piece of data:
a coherent form $C$. This is a closed $3$-form on
the Lie algebroid $\ample$ satisfying certain compatibility
conditions. In this case $(\ample,C)$ is called a
characteristic pair.

An equivalence is introduced on characteristic pairs. Let
$H^{\bullet}_{\leftrightarrow}(\Abd)$ denote  the cohomology groups
of the subcomplex $\subcomplex^\bullet(\Abd)$ of
$(\sections{\wedge^\bullet\Abd^*},\dee)$, where
$\subcomplex^k(\Abd)$ is made of the sections of $\wedge^k\Abd^*$
which are annihilated by all sections of $\wedge^k\GG$. Roughly
speaking, two coherent forms $C_1$ and $C_2$ on $\ample$ are
equivalent if and only if the class of $C_1-C_2$ in
$H^3_{\leftrightarrow}(\Abd)$ is zero. We prove that
there is a one-to-one correspondence between regular Courant
algebroids up to isomorphisms and equivalence classes of
characteristic pairs.

The inclusion $i:\GG\to\ample$ is a Lie algebroid morphism.
Therefore, it induces a morphism $i^*:H^3(\ample)\to H^3(\GG)$.
We denote the class of the well-known Cartan $3$-form of Lie theory in $H^3(\GG)$ by $\alpha$.

The main result of the paper can be summarized as the following

\begin{thm*}
\begin{enumerate}
\item There is a natural map
from regular Courant algebroids to
 quadratic Lie algebroids
 with vanishing first Pontryagin class.
\item
For any quadratic Lie algebroid $\ample$ with vanishing first
Pontryagin class, the isomorphism classes of Courant
algebroids whose ample Lie algebroids are isomorphic
to $\ample$ are  parametrized by
$\frac{\anchorUps H^3(F)}{\mathbb{I}}\cong
\frac{H^3(F)}{(\anchorUps)\inverse(\mathbb{I})}$.
 Here $\mathbb{I}$
is a certain  abelian subgroup of  $a^* H^3(F)\subset
H^3_{\leftrightarrow}(\Abd)$ and $a$ denotes the anchor of $\ample$.

\item For any regular Courant
algebroid $E$, there is a degree 3 characteristic class, which is
a cohomology class in  $(i^*)^{-1}(\alpha)\subset H^3(\Abd)$.
 Here  $\ample$ is the ample Lie algebroid of $E$,
and $\alpha\in
H^3(\GG)$ is the Cartan $3$-form.
\end{enumerate}
\end{thm*}

Hence, we have the ``exact sequence''
\begin{center}
\fbox{\begin{minipage}{45mm} isomorphism classes of regular Courant algebroids with characteristic distribution $F$ \end{minipage}} $\xto{\Phi}$
\fbox{\begin{minipage}{45mm} isomorphism classes of quadratic Lie algebroids with characteristic distribution $F$ \end{minipage}} $\xto{\text{FPC}}$
\fbox{$H^4(F)$} .
\end{center}

To this day, little is known about Courant algebroid cohomology~\cite{MR1958835}: Roytenberg computed it for $TM\oplus T^*M$ and Ginot \& Gr\"utzmann for transitive and some very special regular Courant algebroids~\cite{GG}. Our result should be useful for computing the Courant algebroid cohomology of arbitrary regular Courant algebroids.

{\bf Acknowledgments}
The authors thank Camille Laurent-Gengoux, Zhang-Ju Liu,
 Jim Stasheff and Alan Weinstein for useful discussions and comments.
  In  particular, we are indebted to Pavol \v{S}evera,
who provides  Example \ref{Rmk:Severa}, and
called their attention to \cite{Severa},
which led to the spotting of  an error in an early version of the paper.

\section{Preliminaries}

\subsection{Regular Courant algebroids}

A Courant algebroid consists of a vector bundle $ E\to M$, a
fiberwise nondegenerate pseudo-metric $\ip{\qm}{\qm}$, a bundle
map $\rho:E\to TM$ called anchor and a $\RR$-bilinear operation
$\db{}{}$ on $\sections{E}$ called Dorfman bracket, which, for all
$f\in\cinf{M}$ and $e_1,e_2,e_3\in\sections{E}$ satisfy the
following relations:
\begin{align}\label{algian:CAaxioms1}
& \db{e_1}{(\db{e_2}{e_3})}=
\db{(\db{e_1}{e_2})}{e_3}+\db{e_2}{(\db{e_1}{e_3})};
\\ \label{algian:CAaxioms2} &
\rho(\db{e_1}{e_2})=\lb{\rho(e_1)}{\rho(e_2)};
\\ \label{algian:CAaxioms3}
& \db{e_1}{(fe_2)}=\big(\rho(e_1)f\big)e_2+f(\db{e_1}{e_2});
\\ \label{algian:CAaxioms4}
& \db{e_1}{e_2}+\db{e_2}{e_1}=2\DD\ip{e_1}{e_2};
\\ \label{algian:CAaxioms5}
& \db{\DD f}{e_1}=0;  \\ \label{algian:CAaxioms6} &
\rho(e_1)\ip{e_2}{e_3}=\ip{\db{e_1}{e_2}}{e_3}+\ip{e_2}{\db{e_1}{e_3}},
\end{align} where $\DD:\cinf{M}\to\sections{E}$ is the
$\RR$-linear map defined by
\[ \ip{\DD f}{e}=\thalf\rho(e)f .\label{coucou}\]

The symmetric part of the Dorfman bracket is given by Eq.
\eqref{algian:CAaxioms4}. The Courant bracket is defined as the
skew-symmetric part
$\cb{e_1}{e_2}=\thalf(\db{e_1}{e_2}-\db{e_2}{e_1})$ of the Dorfman
bracket. Thus we have the relation
$\db{e_1}{e_2}=\cb{e_1}{e_2}+\DD\ip{e_1}{e_2}$.

Using the identification $\Xi:E\to E^*$ induced by the pseudo-metric $\ip{\qm}{\qm}$:
\[ \duality{\Xi(e_1)}{e_2}\defbe \ip{e_1}{e_2},\quad\forall e_1,e_2 \in E, \]
we can rewrite Eq.~\eqref{coucou} as
\[ \DD f=\thalf\Xi\inv\rho^* df .\]
It is easy to see that $(\ker\rho)\ortho$, the subbundle of $E$
orthogonal to $\ker\rho$ w.r.t.\ the pseudo-metric, coincides with
$\Xi\inv\rho^*(T^*M)$, the subbundle of $E$ generated by the image
of $\DD:\cinf{M}\to\sections{E}$. From Eq.~\eqref{algian:CAaxioms6},
it follows that $\rho(\DD f)=0$ for any $f\in\cinf{M}$. Therefore,
the kernel of the anchor is coisotropic: \[
(\ker\rho)\ortho\subset\ker\rho .\]

The spaces of sections of $\ker\rho$ and $(\ker\rho)\ortho$ are two-sided ideals of $\sections{E}$ w.r.t.\ the Dorfman bracket.

A Courant algebroid $E$ is said to be \textbf{regular} if $F:=\rho(E)$ has constant rank, in which case $F$ is an integrable distribution on the base manifold $M$.
Moreover, if $E$ is regular, then $\ker\rho$ and $(\ker\rho)\ortho$ are smooth (constant rank) subbundles of $E$ and the quotients $E/\ker\rho$ and $E/(\ker\rho)\ortho$ are Lie algebroids.
Obviously, $E/\ker\rho$ and $F$ are canonically isomorphic.
We call $E/(\ker\rho)\ortho$ the \textbf{ample} Lie algebroid associated to $E$.
It will be denoted by the symbol $\ample_E$.

The inclusions
\[ (\ker\rho)\ortho\subset\ker\rho\subset E \]
yield four exact sequences:
\[ \xymatrix{
0 \ar[dr] \ar@(ur,ul)[rr] & & (\ker\rho)\ortho \ar[dr] \ar@(ur,ul)[rr] & & E \ar[dr]^q
\ar@(ur,ul)[rr]^{\rho} & & \frac{E}{\ker\rho} \ar[dr] \ar@(ur,ul)[rr] & & 0 \\
& 0 \ar[ur] \ar[dr] & & \ker\rho \ar[ur] \ar[dr]_{\pi} & & \frac{E}{(\ker\rho)\ortho} \ar[ur] \ar[dr] & & 0 \ar[ur] \ar[dr] & \\ 0 \ar[ur] \ar@(dr,dl)[rr] & & 0 \ar[ur] \ar@(dr,dl)[rr]
& & \GG \ar[ur] \ar@(dr,dl)[rr] & & 0 \ar[ur] \ar@(dr,dl)[rr] & & 0
} \]

\vspace{5mm} Here $\GG:=\ker\rho/(\ker\rho)\ortho$. We use the
symbol $\pi$ (resp. $q$) to denote the projection $\ker\rho\to\GG$
(resp. $E\to\ample_E=E/(\ker\rho)\ortho$).

\subsection{Bundle of quadratic Lie algebras}

A $\cinf{M}$-bilinear and skew-symmetric bracket on $\sections{\GG}$ determined by the Dorfman bracket of $E$ through the relation
\[ \lb{\pi(r)}{\pi(s)}^{\GG}=\pi(\db{r}{s}) ,\qquad\forall r,s\in\sections{\ker\rho} ,\]
turns $\GG$ into a bundle of Lie algebras.
Moreover, the map \[ \pi(r)\otimes\pi(s)\mapsto\ip{r}{s} \]
is a well-defined nondegenerate symmetric and ad-invariant
pseudo-metric on $\GG$, which we will denote by the symbol $\ipG{\qm}{\qm}$.
Hence $\GG$ is a bundle of quadratic Lie
algebras\footnote{A Lie algebra is said to be quadratic if there
exists a nondegenerate, ad-invariant inner product on its underlying
vector space.}.

Note that $\GG$ is also a module over $\ample_E$; the representation is given by
\[ x\cdot\pi(r) =  \lb{x}{q(r)} , \qquad \forall\; x\in\sections{\ample_E},\; r\in\sections{\ker\rho} .\]
This representation is compatible with the pseudo-metric on $\GG$:
\begin{equation}\label{Prop:fatofEquadratic} a(\ax)\ipg{\gr}{\gs}=\ipg{\ax\cdot\gr}{\gs}+\ipg{\gr}{\ax\cdot\gs}, \qquad \forall \; \ax\in\sections{\ample_E},\; \gr,\gs\in\sections{\GG} .\end{equation}

\subsection{Dissections of regular Courant algebroids}

Let $E$ be a regular Courant algebroid with characteristic distribution $F$ and bundle of quadratic Lie algebras $\GG$. A \textbf{dissection} of $E$ is an isomorphism of vector bundles
\[ \dissection:F^*\oplus\GG\oplus F\to E \]
such that
\[ \ip{\dissection(\xi+\gr+\fx)}{\dissection(\eta+\gs+\fy)}
=\thalf\duality{\xi}{\fy}+\thalf\duality{\eta}{\fx}+\ipg{\gr}{\gs}
,\]
for all $\xi,\eta\in\sections{F^*}$, $\gr,\gs\in\sections{\GG}$ and $\fx,\fy\in\sections{F}$.
Such an isomorphism transports the Courant algebroid structure of $E$ to $F^*\oplus\GG\oplus F$.

Each dissection of $E$ induces three canonical maps:
\begin{enumerate}
\item
$\ConnectionDer:\sections{F}\otimes\sections{\GG}\to\sections{\GG}$:
\begin{equation}\label{Eqt:inducedConnectionDer}
\ConnectionDer_x \gr=\ProjectionTo{\GG}(\db{x}{\gr}), \qquad\forall\; x\in\sections{F},\; \gr\in\sections{\GG};
\end{equation}
\item
$\Curvature:\sections{F}\otimes\sections{F}\to\sections{\GG}$:
\begin{equation}\label{Eqt:inducedCurvature}
\Curvature(x,y)=\ProjectionTo{\GG}(\db{x}{y}), \qquad\forall\; x,y\in\sections{F};
\end{equation}
\item
$\HForm:\sections{F}\otimes\sections{F}\otimes\sections{F}\to\cinf{M}$:
\begin{equation}\label{Eqt:inducedHForm}
\HForm(x,y,z)=\duality{\ProjectionTo{\Fs}(\db{x}{y})}{z}, \qquad\forall\; x,y,z\in\sections{F}.
\end{equation}
\end{enumerate}

\begin{prop}\label{Prop:CourantbracketIntermsofData}
\begin{enumerate}
\item The map $\ConnectionDer$ satisfies
\[ \ConnectionDer_{f\fx}\gr=f\ConnectionDer_{\fx}\gr \qquad \text{and} \qquad
\ConnectionDer_{\fx}(f\gr)=f\ConnectionDer_{\fx}\gr + \big(\fx(f)\big)\gr ,\]
for all $\fx\in\sections{F}$, $\gr\in\sections{\GG}$ and $f\in\cinf{M}$.
\item The map $\Curvature$ is skew-symmetric and $\cinf{M}$-bilinear. It can thus be regarded as a bundle map $\wedge^2F\to\GG$.
\item The map $\HForm$ is skew-symmetric and $\cinf{M}$-bilinear. It can thus be regarded as a section of $\wedge^3\Fs$.
\end{enumerate}
\end{prop}

The following lemma shows that dissections always exist.

\begin{lem}\label{Lem:existenceofsplitting}
Let $(E,\ip{\qm}{\qm},\db{}{},\rho)$ be a regular Courant algebroid. And set $F=\rho(E)$.
\begin{enumerate}
\item There exists a splitting $\splitting:F\to E$ of the short exact sequence
\begin{equation}\label{Sequence:KerrhoEF} 0\to\ker\rho\to E\xto{\rho}F\to 0 \end{equation}
whose image $\splitting(F)$ is isotropic in $E$.
\item Given such a splitting $\splitting$, there exists a unique splitting $\sigma_{\splitting}:\GG\to\ker\rho$ of the short exact sequence
\begin{equation}\label{Sequence:ImgrhoUpsKerrhoGap}
0\to(\ker\rho)\ortho\to\ker\rho\xto{\pi}\GG\to 0
\end{equation}
with image $\sigma_{\splitting}(\GG)$ orthogonal to $\splitting(F)$ in $E$.
\item Given a pair of splittings $\splitting$ and $\sigma_{\splitting}$ as above,
the map $\dissection_{\splitting}:F^*\oplus\GG\oplus F\to E$ defined by
\[ \dissection_{\splitting}(\xi+\gr+\fx)
=\tfrac{1}{2}\Xi\inv\rho^*(\xi)+\sigma_{\splitting}(\gr)+\splitting(\fx) \]
is a dissection of $E$.
\end{enumerate}
\end{lem}

\begin{proof}
(a) Take any section $\lambda_0$ of $\rho:E\to F$.
Consider the bundle map $\varphi:F\to F^*$ defined by
\[ \duality{\varphi(\fx)}{\fy}=\ip{\splitting_0(\fx)}{\splitting_0(\fy)} ,\]
where $\fx,\fy\in\sections{F}$.
Then $\splitting=\splitting_0-\thalf\Xi\inv\rho^*\varphi$ is a section of $\rho:E\to F$
such that \[ \ip{\splitting{\fx}}{\splitting{\fy}}=0, \qquad \forall \fx,\fy\in\sections{F} .\]
(b) Take any section $\sigma_0$ of $\pi:\ker\rho\to\GG$.
Consider the bundle map $\psi:\GG\to F^*$ defined by
\[ \duality{\psi(\gr)}{\fx}=\ip{\sigma_0(\gr)}{\splitting(\fx)} ,\]
where $\gr\in\GG$ and $\fx\in F$.
Then $\sigma_{\splitting}=\sigma_0-\Xi\inv\rho^*\psi$ is the only section of $\pi:\ker\rho\to\GG$
such that
\[ \ip{\sigma_{\splitting}{\gr}}{\splitting{\fx}}=0, \qquad \forall\;\gr\in\GG,\;\fx\in\sections{F} .\]
(c) This is obvious.
\end{proof}


\subsection{Covariant derivatives on regular Courant algebroids}
\label{Sec:covderregca}

Let us first recall the notion of $E$-connections~\cite{AlekseevXu}.
Vaisman's metric connections are a related notion~\cite{MR2178250}.

\begin{defn}An $E$-connection (or $E$-covariant derivative)
 on $E$ is an $\RR$-bilinear map
\[ \Econnection:\sections{E}\otimes\sections{E}\to
\sections{E}:(e,s)\mapsto\Econnection_{e}s \]
such that
\begin{gather*}
\Econnection_{fe}s = f\Econnection_{e}s ,\\
\Econnection_{e}(fs) = f\Econnection_{e}s+(\rho(e)f)s ,
\end{gather*}
for all $f\in\cinf{M}$ and $e, s\in\sections{E}$.
\end{defn}

The next lemma shows that each regular Courant algebroid $E$ admits $E$-connections.

If $E$ is a regular Courant algebroid, there always exists a
torsion-free connection $\nabla^F$ on the integrable distribution
$F=\rho(E)\subset TM$.
Indeed, it suffices to consider the restriction to $F$ of the Levi-Civit\`a connection of some Riemanian metric on $M$. In the sequel, the symbol $\nabla^F$ will be used to denote a chosen torsion-free connection on $F$ and the dual connection on $F^*$ as well.

\begin{lem}
\begin{enumerate}
\item Given a dissection $\dissection$ of a regular Courant algebroid $E$
(identifying $E$ with $F^*\oplus\GG\oplus F$)
and the data $\ConnectionDer$, $\Curvature$ and $\HForm$
it induces, each torsion-free connection $\nabla^F$ on $F$
determines an $E$-connection $\Econnection$ on $E$ through the defining relation
\begin{equation}\label{Eqn:Econnectionconcstructed}
\Econnection_{\xi+\gr+\fx}(\eta+\gs+\fy)
=(\Fconnection_\fx\eta- \tfrac{1}{3} \HForm(\fx,\fy,\qm)+
(\ConnectionDer_\fx \gs +\tfrac{2}{3} \lbG{\gr}{\gs}) + \Fconnection_\fx \fy
,\end{equation}
where $\fx,\fy\in\sections{F}$, $\xi,\eta\in\sections{\Fs}$ and $\gr,\gs\in\sections{\GG}$.
\item This covariant derivative $\Econnection$ preserves the pseudo-metric:
\[ \rho(e_1)\ip{e_2}{e_3}=\ip{\Econnection_{e_1}{e_2}}{e_3}+\ip{e_2}{\Econnection_{e_1}{e_3}},
\quad\forall e_1,e_2,e_3\in\sections{E} .\]
\item Moreover, we have
\[ \rho(\Econnection_{e_1}{e_2}-\Econnection_{e_2}{e_1})=\lb{\rho(e_1)}{\rho(e_2)} .\]
\end{enumerate}
\end{lem}

We omit the proof as it is straightforward.

\subsection{Naive cohomology}
\label{naive_coho}

Let $E$ be a Courant algebroid.
The nondegenerate pseudo-metric on $E$ induces
a nondegenerate pseudo-metric on $\wedge^k E$:
\[ \ip{e_1\wedge\cdots\wedge e_k}{f_1\wedge\cdots\wedge f_k}
= \left| \begin{array}{cccc}
\ip{e_1}{f_1} & \ip{e_1}{f_2} & \cdots & \ip{e_1}{f_k} \\
\ip{e_2}{f_1} & \ip{e_2}{f_2} & \cdots & \ip{e_2}{f_k} \\
\vdots & \vdots & \ddots & \vdots \\
\ip{e_k}{f_1} & \ip{e_k}{f_2} & \cdots & \ip{e_k}{f_k}
\end{array} \right| .\]
The isomorphism of vector bundles $\wedge^k E\to\wedge^k E^*$
coming from this nondegenerate pseudo-metric will be denoted
by the same symbol $\Xi$. In the sequel, $\wedge^k E$ and
$\wedge^k E^*$ are sometimes identified with each other
when it is clear from the context.

The sections of $\wedge^k \ker\rho$ are called naive $k$-cochains.
The operator
\[ \dsmile:\Gamma (\wedge^k\ker\rho)\to \Gamma (\wedge^{k+1}\ker\rho ) \]
defined by the relation
\begin{multline}\label{m2}
\ip{\dsmile s}{e_0\wedge e_1\wedge\cdots\wedge e_k}
=\sum_{i=0}^k (-1)^i \rho(e_i) \ip{s}{e_0\wedge\cdots\wedge\widehat{e_i}\wedge\cdots\wedge e_k} \\
+ \sum_{i<j} (-1)^{i+j} \ip{s}{\cb{e_i}{e_j}\wedge e_0\wedge \cdots\wedge \widehat{e_i}\wedge
\cdots\wedge \widehat{e_j}\wedge \cdots,e_k} ,
\end{multline}
where $s\in\sections{\wedge^k\ker\rho}$ and $e_0,\dots,e_k\in\sections{E}$,
makes $(\sections{\wedge^{\bullet}\ker\rho},\dsmile)$ a cochain complex.
Its cohomology $H^\bullet_{\mathrm{naive}}(E)$ is called the
\textbf{naive cohomology} of the Courant algebroid $E$~\cite{MathieuXu}.

\begin{rmk}\label{Rmk:NaiveIdFormsOnFat}
If $E$ is regular, then $\ample_{E}=E/(\ker\rho)\ortho$ is a regular
Lie algebroid. And it is easy to see that the cochain complexes
$(\sections{\wedge^{\bullet}\ker\rho},\dsmile)$ and
$(\sections{\wedge^{\bullet}\ample_E^*},\dee)$ are isomorphic.
Indeed, we have
\[ \Xi(\ker\rho)=\big((\ker\rho)\ortho\big)^0=q^*(\ample_E^*) \]
and \[ (\Xi\inv\rond q^*)\rond \dee = \dsmile\rond (\Xi\inv\rond
q^*) .\] Hence the naive cohomology of $E$ is isomorphic to the Lie
algebroid cohomology of $\ample_{E}$.
\end{rmk}

\subsection{A degree-$3$ characteristic class}
\label{Sec:main1-3class}

The map $K:\sections{\wedge^3 E}\to\RR$ defined by
\[ K(e_1,e_2,e_3) = \ip{\cb{e_1}{e_2}}{e_3}+\CP ,\]
where $e_1,e_2,e_3\in\sections{E}$, is \emph{not} a $3$-form on $E$
as it is not $\cinf{M}$-linear.

However, it can be modified using an $E$-connection so that the result is $\cinf{M}$-linear.

\begin{lem}[\cite{AlekseevXu}]
 If $\Econnection$ is a covariant derivative on a Courant algebroid $E$, then
\begin{equation}\label{AlexXu3form}
C_{\Econnection}(e_1,e_2,e_3)=\tfrac{1}{3}\ip{\cb{e_1}{e_2}}{e_3}-\thalf
\ip{\Econnection_{e_1}e_2-\Econnection_{e_2}e_1}{e_3}+\CP ,
\end{equation}
where $e_1,e_2, e_3\in \sections{E}$, defines a $3$-form on $E$.
\end{lem}

%

The following is an analogue of the Chern-Weil construction.

\begin{thm}\label{Thm:main1}
Let $E$ be a regular Courant algebroid.
\begin{enumerate}
\item If $\Econnection$ is the $E$-connection on $E$ given by~\eqref{Eqn:Econnectionconcstructed},
then the $3$-form $C_{\Econnection}$ does not depend on the chosen torsion-free connection $\nabla^F$.
It only depends on the chosen dissection $\dissection$.
It will henceforth be denoted by the symbol $C_{\dissection}$.
\item The $3$-form $C_{\dissection}$ is a naive $3$-cocycle.
\item The cohomology class of $C_{\dissection}$ does not depend on the chosen dissection $\dissection$.
\end{enumerate}
\end{thm}

The class of $C_{\dissection}$ is called the \textbf{characteristic class} of the Courant algebroid $E$.
The proof of this theorem is deferred to Section~\ref{Sec:proofs}.

\subsection{First Pontryagin class of a quadratic Lie algebroid}
\label{FPC}

We fix a regular Lie algebroid $(\Abd,\algebroidlb{\qm}{\qm},\anchor)$
over the base manifold $M$ and with cohomology differential operator $\dee:\sections{\wedge^{\bullet}\Abds}\to\sections{\wedge^{\bullet+1}\Abds}$.

We have the short exact sequence of vector bundles
\[ 0\to \GG\to \Abd\xto{\anchor} F\to 0 ,\]
where $F=\anchor(\Abd)$ and $\GG=\ker\anchor$.

It is clear that $\GG$ is a bundle of Lie algebras, called the Lie algebra bundle of $\Abd$.
The fiberwise bracket is denoted by $\lbG{\qm}{\qm}$.

\begin{defn}
\label{def:quad}
\begin{enumerate}
\item A regular Lie algebroid $\Abd$ is said to be a \textbf{quadratic} Lie algebroid if the kernel $\GG$ of its anchor $\anchor$ is equipped with a fiberwise nondegenerate $ad$-invariant symmetric bilinear form
$\ipG{\qm}{\qm}$ satisfying:
\[ \anchor{(X)}\ipG{\gr}{\gs}=\ipG{\algebroidlb{X}{\gr}}{\gs}+\ipG{\gr}{\algebroidlb{X}{\gs}},
\qquad \forall X\in \sections{\Abd}, \gr,\gs\in\sections{\GG} .\]
\item Two regular quadratic Lie algebroids $\Abd_1$ and $\Abd_2$ are said to be isomorphic if there is a Lie algebroid isomorphism $\Abd_1\to\Abd_2$, whose restriction $\GG_1\to\GG_2$ to the kernels of the anchors is an isomorphism of quadratic Lie algebras.
\end{enumerate}
\end{defn}

For example, the ample Lie algebroid $\ample_{E}=E/(\ker\rho)\ortho$ of
any regular Courant algebroid $(E,\ip{\qm}{\qm},\db{}{},\rho)$ is a quadratic Lie algebroid.

Now a natural question arises as to whether every quadratic Lie algebroid
can be realized as the ample Lie algebroid of a Courant algebroid.
It turns out that there is an obstruction. To see this, let us introduce some notations.

A \textbf{hoist} of the Lie algebroid $\Abd$ is a section $\hoist:F\to\Abd$ of the anchor $\anchor$.
The bundle map $\hoistcurvature:\wedge^2 F\to\GG$ given by
\begin{equation}\label{Eqt:Curvature}
\hoistcurvature(x,y)=\algebroidlb{\hoist(x)}{\hoist(y)}-\hoist([x,y]), \qquad\forall
x,y\in\sections{F}
\end{equation}
is its \textbf{curvature}.

Given a quadratic Lie algebroid endowed with a hoist,
we can consider the $4$-form
$\PontryaginCocycle\in\sections{\wedge^4 F^*}$ given by
\[ \PontryaginCocycle(x_1,x_2,x_3,x_4)=
\tfrac{1}{4}\sum_{\sigma\in S_4} \sgn(\sigma)
\ipg{\hoistcurvature(x_{\sigma(1)},x_{\sigma(2)})}
{\hoistcurvature(x_{\sigma(3)},x_{\sigma(4)})} ,\]
for all $x_1,x_2,x_3,x_4\in\sections{F}$.

\begin{lem}
The $4$-form $\PontryaginCocycle$ is closed and its cohomology class
in $H^4(F)$ does not depend on the choice of $\hoist$.
\end{lem}

Following~\cite{Severa, MR2360313}, this cohomology class is called
the \textbf{first Pontryagin class} of the quadratic Lie algebroid $\Abd$.

\begin{thm}\label{Thm:Bressler'sThm}
A regular quadratic Lie algebroid is isomorphic
to the ample Lie algebroid associated to some regular Courant algebroid
if and only if its first Pontryagin class vanishes.
\end{thm}

The proof of this theorem is deferred to Section~\ref{Sec:standard3form}.

\subsection{Characteristic pairs and regular Courant algebroids}

\begin{defn}\label{Defn:good}
\begin{enumerate}
\item Let $\Abd$ be a regular quadratic Lie algebroid and let $\hoist$ be a hoist of $\Abd$.
A closed $3$-form $C\in\Gamma(\wedge^3\Abd^*)$ is said to be \textbf{$\hoist$-coherent} if
\begin{gather}
\label{Align:GoodformsCdt1} C(\gr,\gs,\gt)=-\ipG{\lbG{\gr}{\gs}}{\gt} ;\\
\label{Align:GoodformsCdt2} C(\gr,\gs,\Connection(x))=0 ;\\
\label{Align:GoodformsCdt3} C(\gr,\Connection(x),\Connection(y))=\ipG{\gr}{\Curvature^{\Connection}(x,y)},
\end{gather}
for all $\gr,\gs,\gt\in\GG$ and $x,y\in F$. We also say that
$(C,\hoist)$ is a coherent pair.
\item A closed $3$-form $C\in\Gamma(\wedge^3\Abd^*)$ on a regular quadratic Lie algebroid is called \textbf{coherent} if
there exists a hoist $\hoist$ such that $C$ is $\hoist$-coherent.
\item A \textbf{characteristic pair} is a couple $(\Abd,C)$ made of a regular
quadratic Lie algebroid $\Abd$ and a coherent $3$-form  $C$ on it.
\end{enumerate}
\end{defn}

\begin{ex}
If $\Abd=\GG$ is a quadratic Lie algebra, then the only \good\ $3$-form
is the Cartan $3$-form defined by Eq.~\eqref{Align:GoodformsCdt1}.
If $\Abd=\GG\times TM$ is the product Lie algebroid
with the obvious quadratic Lie algebroid structure
and $\HForm$ is a closed $3$-form on $M$,
then the $3$-form ${C}\in\sections{\wedge^3\Abd^*}$ defined by
\[ {C}(\gr+x,\gs+y, \gt+z) = -\ipG{\lbG{\gr}{\gs}}{\gt}+\HForm(x,y,z) ,\]
for all $\gr,\gs,\gt\in\sections{\GG}$ and $x,y,z\in\sections{TM}$,
is \good.
\end{ex}

\begin{prop}\label{Prop:PontryaginCoherentExtensionAllEquivalent}
Let $\Abd$ be a regular quadratic Lie algebroid.
 The following statements are equivalent.
\begin{enumerate}
\item[1)] The first Pontryagin class of $\Abd$ vanishes.
\item[2)] There exists a coherent $3$-form on $\Abd$.
\item[3)] There exists a Courant algebroid whose ample Lie algebroid is  $\Abd$.
\end{enumerate}\end{prop}

Consider the subcomplex $\subcomplex^\bullet(\Abd)$ of $(\sections{\wedge^\bullet\Abd^*},\dee)$,
where $\subcomplex^k(\Abd)$ is made of the sections of $\wedge^k\Abd^*$
which are annihilated by all sections of $\wedge^k\GG$.
Its cohomology groups are denoted $H^{\bullet}_{\leftrightarrow}(\Abd)$.
Given two coherent $3$-forms $C_1$ and $C_2$ on $\Abd$, we have $C_2-C_1\in\subcomplex^3(\Abd)$.

\begin{defn}
Two characteristic pairs $(\Abd_1,C_1)$ and $(\Abd_2,C_2)$ are said to be \textbf{equivalent} if there exists an isomorphism
of quadratic Lie algebroids $\sigma:\Abd_1\to\Abd_2$ such that $[C_1-\sigma^* C_2]=0\in H^3_{\leftrightarrow}(\Abd_1)$.
The equivalence class of a characteristic pair $(\Abd,C)$ will be denoted $\class{(\Abd,C)}$.
\end{defn}

Our main result is the following:

\begin{thm}\label{Thm:onetoone2}
\begin{enumerate}
\item
There is a one-to-one correspondence between regular Courant algebroids
up to isomorphisms and equivalence classes of characteristic pairs.
\item
If $E$ is a Courant algebroid corresponding to the
characteristic pair $(\Abd, C)$, then
the characteristic class of $E$ is equal to $[C]\in H^3(\Abd )$.
\end{enumerate}
\end{thm}

The proof is postponed to Section~\ref{Sec:proofs}.

\section{Standard structures}

\subsection{Standard Courant algebroid structures on $\Fs\oplus \GG\oplus F$}

Let $F$ be an integrable subbundle of $TM$ and $\GG$ be a bundle of quadratic Lie algebras over $M$.

In this section, we are only interested in those Courant algebroid structures on $\fsgf$ whose anchor map is
\begin{equation} \label{Eqt:Standardrho} \rho(\xi_1+\gr_1+x_1)=x_1 ,\end{equation}
whose pseudo-metric is
\begin{equation} \label{Eqt:Standardip}
\ip{\xi_1+\gr_1+\fx_1}{\xi_2+\gr_2+\fx_2} = \thalf\duality{\xi_1}{\fx_2}+\thalf\duality{\xi_2}{\fx_1}+\ipG{\gr_1}{\gr_2} ,\end{equation}
and whose Dorfman bracket satisfies
\begin{equation} \label{Eqt:prGdb} \ProjectionTo{\GG} (\db{\gr_1}{\gr_2})=\lbG{\gr_1}{\gr_2} ,\end{equation}
where $\xi_1,\xi_2\in\Fs$, $\gr_1,\gr_2\in\GG$, and $\fx_1,\fx_2\in F$.
We call them \textbf{standard} Courant algebroid structures on $\fsgf$.

Given such a standard Courant algebroid structure on $\fsgf$, we define
$\ConnectionDer$, $\Curvature$, and $\HForm$ as in Eqs.~\eqref{Eqt:inducedConnectionDer}-\eqref{Eqt:inducedHForm}.
The following Lemma shows that the Dorfman bracket on $\sections{\fsgf}$ can be recovered from $\ConnectionDer$, $\Curvature$, and $\HForm$.

\begin{lem} \label{Prop:CourantbracketIntermsofData2}
Let $\PForm:\sections{\GG}\otimes\sections{\GG}\to\sections{\Fs}$ and
$\mathcal{Q}: \sections{F}\otimes\sections{\GG}\to
\sections{F^*}$ be the maps defined by
\begin{equation} \label{Eqt:definationofPForm}
\duality{\PForm(\gr_1,\gr_2)}{y} = 2\ipG{\gr_2}{\ConnectionDer_y\gr_1} \end{equation}
and
\begin{equation} \label{Eqt:definationofQ}
\duality{\mathcal{Q}(x,\gr)}{y} = \ipG{\gr}{\Curvature(x,y)} .\end{equation}
Then we have
\begin{align}
\label{FSGFx1x2}
&\db{\fx_1}{\fx_2}=\HForm(\fx_1,\fx_2,\qm)
+\Curvature(\fx_1,\fx_2)+\lb{\fx_1}{\fx_2}, \\ \label{FSGFg1g2}
&\db{\gr_1}{\gr_2}=\PForm
(\gr_1,\gr_2)+\lbG{\gr_1}{\gr_2}, \\ \label{FSGFxi1g2}
&\db{\xi_1}{\gr_2}=\db{\gr_1}{\xi_2}=\db{\xi_1}{\xi_2}=0, \\ \label{FSGFx1xi2}
&\db{\fx_1}{\xi_2}=\LieDer_{\fx_1}\xi_2, \\ \label{FSGFxi1x2}
&\db{\xi_1}{\fx_2}=-\LieDer_{\fx_2}\xi_1+
\dF\duality{\xi_1}{\fx_2}, \\ \label{FSGFx1g2}
&\db{\fx_1}{\gr_2}=-\db{\gr_2}{\fx_1}=-2\mathcal{Q}(\fx_1,\gr_2)+
\ConnectionDer_{\fx_1}\gr_2,
\end{align}
for all $\xi_1,\xi_2\in\sections{\Fs}$, $\gr_1,\gr_2\in\sections{\GG}$,
$\fx_1,\fx_2\in\sections{F}$.
Here $\dF:\cinf{M}\to\sections{F^*}$ denotes the leafwise de~Rham differential.
\end{lem}

\begin{proof}
Eqs.~\eqref{FSGFx1x2} and~\eqref{FSGFxi1g2} are quite
obvious. To get Eq.~\eqref{FSGFg1g2}, we observe that, for all
$x\in\sections{F}$,
\begin{multline*}
\duality{\ProjectionTo{\Fs}(\db{\gr_1}{\gr_2})}{x}
= 2\ip{\ProjectionTo{\Fs}(\db{\gr_1}{\gr_2})}{x}
= 2\ip{\db{\gr_1}{\gr_2}}{x}
= -2\ip{\gr_2}{\db{\gr_1}{x}} \\
= 2\ip{\gr_2}{\db{x}{\gr_1}}
= 2\ip{\gr_2}{\ProjectionTo{\GG}(\db{x}{\gr_1})}
= 2\ip{\gr_2}{\ConnectionDer_{x}{\gr_1}}.
\end{multline*}
Eq.~\eqref{FSGFx1g2} can be proved similarly. To prove Eq.~\eqref{FSGFx1xi2},
we only need to show that
$\ProjectionTo{\GG}(\db{x_1}{\xi_2})=0$. In fact, for all
$\gr\in\sections{\GG}$,
\begin{equation*}
\ipG{\ProjectionTo{\GG}(\db{x_1}{\xi_2})}{\gr}
= \ip{\db{x_1}{\xi_2}}{\gr}
= \LieDer_{x_1}\ip{\xi_2}{\gr}-\ip{\xi_2}{\db{x_1}{\gr}}
= -\thalf\duality{\xi_2}{\ProjectionTo{F}(\db{x_1}{\gr})}
= 0.
\end{equation*}
Eq.~\eqref{FSGFxi1x2} follows directly from Eq.~\eqref{FSGFx1xi2}.
\end{proof}

\begin{prop}\label{Prop:CourantbracketIntermsofData3}
The following identities hold:
\begin{gather}
\label{ipGinvariant} \ld{x}\ipG{\gr}{\gs}=\ipG{\ConnectionDer_{x}{\gr}}{\gs}+\ipG{\gr}{\ConnectionDer_{x}{\gs}}, \\
\label{lbGinvariant} \ConnectionDer_{x}\lbG{\gr}{\gs}=\lbG{\ConnectionDer_x\gr}{\gs}+\lbG{\gr}{\ConnectionDer_x\gs}, \\
\label{dCurvautre0} \big(\ConnectionDer_x\Curvature(y,z)-\Curvature(\lb{x}{y},z)\big)+\CP=0, \\
\label{CurvatureConnectionDer} \ConnectionDer_x\ConnectionDer_y\gr-\ConnectionDer_y\ConnectionDer_x\gr
-\ConnectionDer_{\lb{x}{y}}\gr=\lbG{\Curvature(x,y)}{\gr},
\end{gather}
for all $x,y,z\in \sections{F}$ and $\gr,\gs\in\sections{\GG}$.

Moreover, we have
\begin{equation} \label{dMHPontryagin} \dF\HForm=\PontryaginCocycleR ,\end{equation}
where $\PontryaginCocycleR$ denotes the $4$-form on $F$ given by
\[ \PontryaginCocycleR(x_1,x_2,x_3,x_4) = \tfrac{1}{4} \sum_{\sigma\in S_4} \sgn(\sigma)
\ipG{\Curvature(x_{\sigma(1)},x_{\sigma(2)})}{\Curvature(x_{\sigma(3)},x_{\sigma(4)})}, \]
where $x_1,x_2,x_3,x_4\in F$.
\end{prop}

\begin{proof}
Applying the expressions found for the Dorfman bracket in Lemma~\ref{Prop:CourantbracketIntermsofData2},
Eq.~\eqref{ipGinvariant} follows from
\[ \LieDer_{x}{\ipG{\gr}{\gs}} = \ipG{\db{x}{\gr}}{\gs} + \ipG{\gr}{\db{x}{\gs}} ,\]
Eq.~\eqref{lbGinvariant} from
\[ \db{x}{(\db{\gr}{\gs})} = \db{(\db{x}{\gr})}{\gs} + \db{\gr}{(\db{x}{\gs})} ,\]
Eqs.~\eqref{dCurvautre0} and~\eqref{dMHPontryagin} from
\[ \db{x}{(\db{y}{z})} = \db{(\db{x}{y})}{z} + \db{y}{(\db{x}{z})} ,\]
and Eq.~\eqref{CurvatureConnectionDer} from
\[ \db{x}{(\db{y}{\gr})} = \db{(\db{x}{y})}{\gr} + \db{y}{(\db{x}{\gr})} .\qedhere \]
\end{proof}

In summary, we have proved the following
\begin{thm}\label{Thm:StandardCourantStructure}
A Courant algebroid structure on $\fsgf$, with pseudo-metric
\eqref{Eqt:Standardip} and anchor map~\eqref{Eqt:Standardrho}, and
satisfying Eq.~\eqref{Eqt:prGdb}, is completely determined by
an $F$-connection $\ConnectionDer$ on $\GG$,
a bundle map $\Curvature:\wedge^2 F\to\GG$,
and a $3$-form $\HForm\in\sections{\wedge^3\Fs}$
satisfying the compatibility conditions~\eqref{ipGinvariant}-\eqref{dMHPontryagin}.
The Dorfman bracket on $\fsgf$ is then given by Eqs.~\eqref{FSGFx1x2}-\eqref{FSGFx1g2}.
\end{thm}

This Courant algebroid structure will be called the standard
Courant algebroid determined by the quintuple $(F,\GG;\ConnectionDer,\Curvature,\HForm)$.

\subsection{Standard quadratic Lie algebroids}\label{Sec:standard3form}

We assume again that $F$ is an integrable
subbundle of $TM$ and $\GG$ is a bundle of quadratic Lie algebras over $M$.

\begin{prop}\label{Prop:StandardQuadraticLieAlgebroidStructure}
Let $\Abd$ be a quadratic Lie algebroid with anchor $\anchor:\Abd\to
TM$ such that $\anchor(\Abd)=F$ and $\ker\anchor $ is isomorphic, as
a bundle of quadratic Lie algebras, to $\GG$. Then any hoist of $\Abd$
determines an isomorphism $\Abd\cong\Abd^s=\GG\oplus F$. The
transported Lie bracket on $\sections{\Abd^s}$ is completely
determined by an $F$-connection $\ConnectionDer$ on $\GG$ and a
bundle map $\Curvature:\wedge^2 F\to\GG$ satisfying the conditions
\eqref{ipGinvariant}-\eqref{CurvatureConnectionDer}, through the
following relations:
\begin{equation*}
\algebroidlb{\gr}{\gs}=\lbG{\gr}{\gs}, \qquad
\algebroidlb{x}{y}=\Curvature({x},{y})+\lb{x}{y}, \qquad
\algebroidlb{x}{\gr}=\ConnectionDer_x\gr,
\end{equation*}
for all $\gr$, $\gs\in \sections{\GG}$, $x$, $y \in\sections{F}$.
\end{prop}

This quadratic Lie algebroid $\Abd^s$ will be called the standard
quadratic Lie algebroid determined by the quadruple
$(F,\GG;\ConnectionDer,\Curvature)$.

Although it is similar to the proof of Corollary~3.2. in~\cite{MR2360313},
we now sketch a proof of Theorem~\ref{Thm:Bressler'sThm} for completeness.

\begin{proof}[Proof of Theorem~\ref{Thm:Bressler'sThm}]
Given a regular Courant algebroid $E$, choose a dissection to
identify it to $\fsgf$. Then, according to
Eq.~\eqref{dMHPontryagin}, the Pontryagin 4-cocycle of the ample Lie
algebroid $\ample_{E}$ associated to $E$ is a coboundary.
Conversely, given a quadratic Lie algebroid $\Abd$ with anchor
$\anchor:\Abd\to TM$ such that $\anchor \Abd =F$ and $\ker\anchor $
is isomorphic to $\GG$ as a bundle of quadratic Lie algebras, assume
there is a hoist $\Connection$ of the Lie algebroid $\Abd$ such that
the Pontryagin $4$-cocycle $\PontryaginCocycle$ is a coboundary, say
$\dee^{F}\HForm$ for some $\HForm\in\sections{\wedge^3\Fs}$. We
define an $F$-connection $\ConnectionDer^{\Connection}$ on $\GG$ by
setting
\begin{equation}\label{Eqn:ConnectionDeruperscribt}
\ConnectionDer^{\Connection}_{x}\gr = \algebroidlb{\Connection(x)}{\gr},
\qquad\forall\;x\in\sections{F},\;\gr\in\sections{\GG}.
\end{equation}
The quintuple $(F,\GG;\ConnectionDer^{\Connection},\Curvature^{\Connection},\HForm)$
satisfies all conditions in Theorem~\ref{Thm:StandardCourantStructure}
and it is clear that $\Abd$ is the ample Lie algebroid associated to the Courant algebroid determined by this quintuple.
\end{proof}

\subsection{Standard $3$-forms on $\Abd^s$ and $E^s$}
\label{Sec:standard3formAbdsEs}

Let $\Abd^s$ denote the standard quadratic Lie algebroid determined by the quadruple
$(F,\GG;\nabla,R)$. Given any $\HForm\in\sections{\wedge^3\Fs}$, we define a $3$-form
${C^s}\in\sections{\wedge^3(\Abd^s)^*}$ by
\begin{equation}\label{labelrouge}
C^s(\gr+\fx,\gs+\fy,\gt+\fz) = \HForm(\fx,\fy,\fz)-\ipG{\lbG{\gr}{\gs}}{\gt}
+\ipG{\Curvature(\fx,\fy)}{\gt}+\ipG{\Curvature(\fy,\fz)}{\gr}+\ipG{\Curvature(\fz,\fx)}{\gs},
\end{equation}
for all $\gr,\gs,\gt\in\sections{\GG}$ and $\fx,\fy,\fz\in\sections{F}$.

\begin{prop}\label{Lem:Csclosediff}
The following statements are equivalent:
\begin{enumerate}
\item $\HForm$ satisfies Eq.~\eqref{dMHPontryagin};
\item $C^s$ is a closed $3$-form;
\item $C^s$ is a \good\ $3$-form.
\end{enumerate}
\end{prop}

\begin{proof}
Straightforward calculations lead to
\begin{align}
\label{align:1} (\dample C^s)(x,y,z,w) =& (\dF \HForm)(x,y,z,w) - \PontryaginCocycleR(x,y,z,w), \\
\label{align:2} (\dample C^s)(x,y,z,\gr) =& \ipG{\big(\ConnectionDer_x\Curvature(y,z) - \Curvature(\lb{x}{y},z)\big) + \CP}{\gr}, \\
\label{align:3} (\dample C^s)(x,y,\gr,\gs) =& 0, \\
\label{align:4} (\dample C^s)(x,\gr,\gs,\gt) =& -\LieDer_x\ipG{\lbG{\gr}{\gs}}{\gt} + \ipG{\lbG{\ConnectionDer_x\gr}{\gs} + \lbG{\gr}{\ConnectionDer_x\gs}}{\gt} + \ipG{\lbG{\gr}{\gs}}{\ConnectionDer_x\gt}, \\
\label{align:5} (\dample C^s)(\gq,\gr,\gs,\gt) =& 0,
\end{align}
for all $x,y,z,w\in\sections{F}$ and $\gq,\gr,\gs,\gt\in\sections{\GG}$.

The right hand side of Eqs.~\eqref{align:2} and~\eqref{align:4} vanish due to Eqs.~\eqref{ipGinvariant}-\eqref{CurvatureConnectionDer}. Thus $\dsmile C^s=0$ if and only if the right hand side of Eq.~\eqref{align:1} vanishes.
The latter is equivalent to Eq.~\eqref{dMHPontryagin}.
This proves the equivalence between the first two statements.
The equivalence between the last two statements follows from Definition~\ref{Defn:good} and Eq.~\eqref{labelrouge}.
\end{proof}

Let $E^s$ be the standard Courant algebroid determined by the quintuple $(F,\GG;\ConnectionDer,\Curvature,\HForm)$.
Clearly its associated ample quadratic Lie algebroid is the standard quadratic Lie algebroid $\Abd^s$ determined by the quadruple
$(F,\GG;\ConnectionDer,\Curvature)$.

It follows from Proposition~\ref{Lem:Csclosediff} and the
compatibility conditions satisfied by
$\ConnectionDer,\Curvature,\HForm$ that the form $C^s$
defined by formula~\eqref{labelrouge} is a closed --- and thus coherent --- $3$-form
on $\Abd^s$. Since the cocycles of the Lie algebroid $\Abd^s$ can be
identified to the naive cocycles on $E^s$, we obtain a naive
$3$-cocycle on $E^s$ called the \textbf{standard $3$-form} on $E^s$.

At this point, Proposition
\ref{Prop:PontryaginCoherentExtensionAllEquivalent} and the
following result  are quite obvious.

\begin{prop}
\label{pro:3formclass} Let $E^s$ be the standard Courant algebroid
determined by the quintuple
$(F,\GG;\ConnectionDer,\Curvature,\HForm)$, and $C^s$ the $3$-form
on $\Abd^s$ defined by~\eqref{labelrouge}. Then the
characteristic class of $E^s$ is equal to $[C^s]\in H^3 (\Abd^s)$.
\end{prop}

\subsection{Isomorphic standard Courant algebroid structures on $\fsgf$}

\begin{prop}\label{GeneralIsomorphism}
Given two different standard Courant algebroid structures $E^s_1$ and $E^s_2$ on $\fsgf$, any isomorphism of Courant algebroids
$\ISO:E^s_1\to E^s_2$ is of the form
\begin{equation}\label{align:ExpressionofI}
\ISO(\xi+\gr+\fx) = \big(\xi+\beta(\fx)-2\varphis\tau(\gr)\big) + \big(\tau(\gr)+\varphi(\fx)\big) + \fx
,\end{equation}
for $\xi\in\Fs$, $\gr\in\GG$ and $\fx\in F$.
Here $\tau$ is an automorphism of the bundle of quadratic Lie algebras $\GG$  and
$\varphi:F\to\GG$ and $\beta:F\to\Fs$ are bundle maps satisfying the following compatibility conditions:
\begin{gather}
\label{Eqn:iso1} \thalf\duality{\beta(\fx)}{\fy}+\thalf\duality{\fx}{\beta(\fy)}+\ipG{\varphi(\fx)}{\varphi(\fy)}=0 ,\\
\label{Eqn:iso2} \ConnectionDer^2_{\fx}\tau(\gr)-\tau\ConnectionDer^1_{\fx}(\gr)=\lbG{\tau(\gr)}{\varphi(\fx)} ,\\
\label{Eqn:iso3}
\Curvature^2(\fx,\fy)-\tau\Curvature^1(\fx,\fy)
= \tau(\ConnectionDer^1_{\fy}\tauinverse\varphi(\fx)
-\ConnectionDer^1_{\fx}\tauinverse\varphi(\fy))
+\varphi\lb{\fx}{\fy}+\lbG{\varphi(\fx)}{\varphi(\fy)}, \\
\label{Eqn:iso4}
\begin{split}
&\quad \HForm^2(\fx,\fy,\fz)-\HForm^1(\fx,\fy,\fz) +2\ipG{\varphi(\fx)}{\lbG{\varphi(\fy)}{\varphi(\fz)}} \\
& = \big(2\ipG{\varphi(\fx)}{\tau\ConnectionDer^1_{\fy}\tauinverse\varphi(\fz)
-\tau\Curvature^1(\fy,\fz)}
+\LieDer_{\fx}\duality{\fy}{\beta(\fz)}+\duality{\fx}{\beta\lb{\fy}{\fz}}\big)+\CP,
\end{split}
\end{gather}
for all $\fx,\fy,\fz\in\sections{F}$ and $\gr\in\sections{\GG}$.
\end{prop}

\begin{proof}
Consider the bundle maps
\begin{align*}
& \tau:\quad \GG \injection E^s_1 \xto{\ISO} E^s_2 \surjection \GG ;\\
& \varphi:\quad F \injection E^s_1 \xto{\ISO} E^s_2 \surjection \GG ;\\
& \beta:\quad F \injection E^s_1 \xto{\ISO} E^s_2 \surjection F^* ;\\
& \gamma:\quad \GG \injection E^s_1 \xto{\ISO} E^s_2 \surjection F^* ,
\end{align*}
where $E^s_i$ stands for $\fsgf$. It is clear that
$\ISO(\fx)=\beta(\fx)+\varphi(\fx)+\fx$, for all $\fx\in F$. Since
$\ISO$ respects the pseudo-metric, we have
$\ip{\ISO(\fx)}{\ISO(\fy)}=0$, for all $\fx,\fy\in F$ and
Eq.~\eqref{Eqn:iso1} follows. Since
$\ISO(\gr)=\gamma(\gr)+\tau(\gr)$ for all $\gr\in\GG$ and
$\ip{\ISO(\gr)}{\ISO(\fx)}=\ip{\gr}{\fx}=0$, we have
$\thalf\duality{\gamma(\gr)}{\fx}+\ipG{\tau(\gr)}{\varphi(\fx)}=0$.
Hence $\gamma=-2\varphis\tau$. From Eq.~\eqref{algian:CAaxioms4}, it
is easy to get that $\ISO(\xi)=\xi$, for all $\xi\in F^*$. Thus
$\ISO$ must be of the form~\eqref{align:ExpressionofI}. Since $\ISO$
takes the Dorfman bracket of $E^s_1$ to $E^s_2$,
Eqs.~\eqref{Eqn:iso2}-\eqref{Eqn:iso4} follow from
Lemma~\ref{Prop:CourantbracketIntermsofData2}.
\end{proof}

The following two technical lemmas follow from direct verifications.

\begin{lem}\label{Lem:dPhiJJ}
Given a bundle map $J:F\to\GG$, define
$\Phi_J\in\subcomplex^2(\Abd^s )$ by
\begin{equation}\label{Eqt:PhiJJ}
\Phi_J(\gr+\fx,\gs+\fy)
=\ipG{\gr}{J(\fy)}-\ipG{\gs}{J(\fx)} ,
\end{equation}
for all $\fx,\fy\in F$ and $\gr,\gs\in\GG$.
Then the differential of $\Phi_J$ is given by:
\begin{multline*}
(\dee\Phi_J)(\gr_1+\fx_1,\gr_2+\fx_2,\gr_3+\fx_3) = \ipG{\gr_1}
{\ConnectionDer_{\fx_3}J(\fx_2)-\ConnectionDer_{\fx_2}J(\fx_3)+J\lb{\fx_2}{\fx_3}} \\
-\ipG{J(\fx_1)}{\Curvature(\fx_2,\fx_3)+\lbG{\gr_2}{\gr_3}}+\CP ,
\end{multline*}
for all $\fx_1,\fx_2,\fx_3\in\sections{F}$ and $\gr_1,\gr_2,\gr_3\in\sections{\GG}$.
\end{lem}

\begin{lem}\label{Lem:dPhiJJbis}
Given a bundle map $K:F\to\Fs$, define
$\Psi_K\in\subcomplex^2(\Abd^s )$ by
\begin{equation}
\Psi_K(\gr+\fx,\gs+\fy)
=\duality{\fx}{K(\fy)}-\duality{\fy}{K(\fx)},
\end{equation}
for all $\fx,\fy\in F$ and $\gr,\gs\in\GG$.
Then the differential of $\Psi_K$ is given by:
\begin{equation*}
\dee\Psi_{K}(\gr+x,\gs+y,\gt+z)
=\big(\duality{x}{\LieDer_{z}K(y)-\LieDer_{y}K(z)+K(\lb{y}{z})}
+\duality{\lb{x}{y}}{K(z)}\big)+\CP ,
\end{equation*}
for all $\fx_1,\fx_2,\fx_3\in\sections{F}$ and $\gr_1,\gr_2,\gr_3\in\sections{\GG}$.
\end{lem}

\begin{prop}\label{Lem:IPullBackDiffCoboundary}
The isomorphism $\ISO:E^s_1\to E^s_2$ of Proposition~\ref{GeneralIsomorphism} yields an isomorphism $\iso:\Abd^s_1\to\Abd^s_2$ between the associated ample Lie algebroid structures on $\GG\oplus F$:
\[ \iso(\gr+\fx)= (\tau(\gr)+\varphi(\fx))+\fx, \qquad\forall\;\gr\in\GG,\;\fx\in F .\]
If $C^s_1$ and $C^s_2$ denote the standard \good\ $3$-forms (defined in Section~\ref{Sec:standard3formAbdsEs}) on $\Abd^s_1$ and $\Abd^s_2$, respectively,
then $\iso^* C^s_2-C^s_1 = \dee(\thalf\Psi_{\beta} + \Phi_{\tau\inv\varphi})$.
\end{prop}

\begin{proof}
The proof is a direct but cumbersome calculation using Proposition~\ref{GeneralIsomorphism}, Lemmas~\ref{Lem:dPhiJJ} and~\ref{Lem:dPhiJJbis}, and Eq.~\eqref{Eqn:iso1}.
\end{proof}

\section{Proofs of the main theorems}\label{Sec:proofs}

\subsection{Proof of Theorem~\ref{Thm:main1}}

For $C_{\Econnection}$ defined in~\eqref{AlexXu3form},
it is easy to see the following equivalent formula
\begin{multline*}
C_{\Econnection}(e_1,e_2,e_3)
= \ip{\db{e_1}{e_2}}{e_3}-\thalf\rho(e_1)\ip{e_2}{e_3}
+\thalf\rho(e_2)\ip{e_3}{e_1}-\thalf\rho(e_3)\ip{e_1}{e_2} \\
-\thalf \ip{\Econnection_{e_1}e_2-\Econnection_{e_2}e_1}{e_3}+\CP
.\end{multline*}
Using the expressions of the Dorfman brackets given by
Proposition~\ref{Prop:CourantbracketIntermsofData2}, one obtains
\begin{multline}\label{Eqn:CstandardonE}
C_{\Econnection}(\xi+\gr+x,\eta+\gs+y,\zeta+\gt+z)
= \ipG{\Curvature(x,y)}{\gt}+\ipG{\Curvature(y,z)}{\gr}+\ipG{\Curvature(z,x)}{\gs} \\
-\ipG{\lbG{\gr}{\gs}}{\gt}+\HForm(x,y,z)
,\end{multline}
for all $\xi,\eta,\zeta\in\sections{\Fs}$,
$\gr,\gs,\gt\in\sections{\GG}$, $x,y,z\in\sections{F}$.  This shows
that $C_{\Econnection}$ coincides with the standard $3$-form $C^s$
on $E^s$ (defined by Eq.~\eqref{labelrouge}). Hence it is naive,
closed, does not depend on the choice of $\Fconnection$ and we get
Statements (1), (2).

Finally, we prove Statement (3). For any two   {dissections }
$\dissection_1$ and $\dissection_2$ of $E$, they induce  two
different Courant algebroid structures on $E^s_i=\Fs\oplus \GG\oplus
F$, but $E^s_1$ and $E^s_2$ are isomorphic. On the other hand,   the
$3$-form $C_{\dissection_i}$ on $E$ is the pull back of a standard
naive $3$-form $({C^s})_i$ on $ E^s_i $. Therefore, by Proposition
\ref{Lem:IPullBackDiffCoboundary} and Remark
\ref{Rmk:NaiveIdFormsOnFat}, $C_{\dissection_1 }-C_{\dissection_2}$
must be a coboundary $\dsmile \varphi$, for some naive $2$-form
$\varphi$ on $E$.

\subsection{Proof of Theorem~\ref{Thm:onetoone2} }

We first show how to construct a Courant
algebroid out of a characteristic pair. Let $\Abd$ be a regular
quadratic Lie algebroid and let $( C, \Connection) $ be a
\good\ pair. Set $F=\anchor(\Abd)$ and $\GG=\ker{\anchor}$.
Together, $\Abd$ and $(C,\Connection)$ induce
the $F$-connection $\ConnectionDer^\Connection$ on $\GG$;
the bundle map $\Curvature^\Connection:\wedge^2 F\to\GG$;
and the $3$-form $\HForm^{(C,\Connection)}\in\sections{\wedge^3\Fs}$
defined by
$\HForm^{(C,\Connection)}(x,y,z)=C(\Connection(x),\Connection(y),\Connection(z))$,
for $x,y,z\in \sections{F}$.
By Proposition~\ref{Lem:Csclosediff}, $C$ being a $\hoist$-\good\ $3$-form
implies that $\HForm^{(C,\Connection)}$ satisfies Eq.~\eqref{dMHPontryagin}
and thus the quintuple
${(F,\GG;\ConnectionDer^\Connection,\Curvature^\Connection,\HForm^{(C,\Connection)})}$
satisfies the requirements to construct
a standard Courant algebroid structure on the bundle
$E^s=\Fs\oplus\GG\oplus F$ described
by Theorem~\ref{Thm:StandardCourantStructure}.
We shall denote such a Courant algebroid by $E^s{(\Abd;C,\Connection)}$.
If $\Abd$ is understood, we just write $E^s{(C,\Connection)}$.

In the sequel, the isomorphism class of a regular Courant algebroid
$E$ will be denoted by $\class{E}$. Similarly, the equivalence class of a
characteristic pair $ {(\Abd,C)}$ is denoted by $\class{(\Abd,C)}$.

Define a map $\Bijectivemapinverse$ from the set of equivalence classes of characteristic
pairs to the set of isomorphism classes of regular Courant algebroids sending
$\class{(\Abd,C)}$ to $\class{E^s{(\Abd;C,\Connection)}}$.
Here $\Connection$ is any hoist of $\Abd$ such that
$(C,\Connection)$ is a \good\ pair on $\Abd$.

\begin{prop}\label{Prop:Bijectivemapinversewelldefined}
The map $\Bijectivemapinverse\big(\class{(\Abd,C)}\big)=\class{E^s{(\Abd;C,\Connection)}}$ is well-defined.
\end{prop}
Before we can prove it, let us   establish some useful facts.
First, we study the change of  hoists. It is obvious that for
any two hoists of $\Abd$: $\Connection$ and $\Connection'$, there
exists a unique map $\DiffConnection:F\to\GG$ such that
$\Connection'(x)=\Connection(x)-\DiffConnection(x)$, $\forall x\in F$.
Recall the notation $\ConnectionDer^{\Connection}$ defined by
formula~\eqref{Eqn:ConnectionDeruperscribt}. It is easy to check that
\begin{gather*}
\ConnectionDer^{\Connection'}_x\gr-\ConnectionDer^{\Connection}_x\gr
=\lbG{\gr}{\DiffConnection(x)}, \\
\Curvature^{\Connection'}(x,y)-\Curvature^{\Connection}(x,y)
=\lb{\DiffConnection(x)}{\DiffConnection(y)}-\ConnectionDer^{\Connection}_x
\DiffConnection(y)+\ConnectionDer^{\Connection}_y \DiffConnection(x)+\DiffConnection[x,y],
\end{gather*}
for all $x,y\in\sections{F}$, $\gr\in\sections{\GG}$.

By Lemma~\ref{Lem:dPhiJJ} and the above two equalities, one can
easily prove the following lemma.
\begin{lem}\label{Lem:goodpairStandardChange}If $(C,\Connection)$ is a
\good\ pair, then so is the pair
$(C+\dee \Phi_{\DiffConnection}, \Connection-\DiffConnection)$. Here
$\Phi_{\DiffConnection}$ is a $2$-form on $\Abd $ defined by
\eqref{Eqt:PhiJJ}.
\end{lem}

An important fact is the following
\begin{lem}\label{Lem:almostidenticalgoodpairsIso}
The Courant algebroids $E^s{(C,\Connection)}$ and
$E^s{(C+\dee\Phi_{\DiffConnection},\Connection-\DiffConnection)}$
are isomorphic.
\end{lem}

\begin{proof} Write $C^1=C$, $C^2=C+\dee
\Phi_{\DiffConnection}$, $\Connection^1=\Connection$,
$\Connection^2=\Connection-\DiffConnection$ and let  $E_i=E^s{( C^i
,\Connection^i)}$. Clearly $E_i$ is determined by the quintuple
${(F,\GG;\ConnectionDer^i,\Curvature^i,\HForm^i)}$, where
$\ConnectionDer^i=\ConnectionDer^{\Connection^i}$,
$\Curvature^i=\Curvature^{\Connection^i}$ and
$\HForm^i=\HForm^{(C^i,\Connection^i)}$.
We have the following equalities:
\begin{gather*}
\ConnectionDer^2_x \gr-\ConnectionDer^1_x \gr=\lbG{\gr}{\DiffConnection(x)}, \\
\Curvature^2(x,y)-\Curvature^1(x,y)=\lbG{\DiffConnection(x)}{\DiffConnection(y)}
-\ConnectionDer^1_x\DiffConnection(y)+\ConnectionDer^1_y\DiffConnection(x)+\DiffConnection\lb{x}{y}, \\
\begin{split} \HForm^2(x,y,z)-\HForm^1(x,y,z)
&= \ipG{\DiffConnection(x)}{\ConnectionDer^1_y\DiffConnection(z)
-\ConnectionDer^1_z\DiffConnection(y)-\DiffConnection\lb{y}{z}-2\Curvature^1(y,z)} \\
&\quad -2\ipG{\lbG{\DiffConnection(x)}{\DiffConnection(y)}}{\DiffConnection(z)} +\CP,
\end{split}
\end{gather*}
for all $x,y,z\in\sections{F}$, $\gr\in\sections{\GG}$.
Define a map $I:E_1\to E_2$ by
\[ I(\xi+\gr+x)=\big(\xi+\beta(x)-2\DiffConnection^*(\gr)\big)+
\big(\gr+\DiffConnection(x)\big)+x ,\]
where $\xi\in\Fs$, $\gr\in\GG$, $x\in F$, and
$\beta:F\to\Fs$ denotes the map $-\DiffConnection^*\circ\DiffConnection$.
According to Proposition~\ref{GeneralIsomorphism} and
the above equalities, $I$ is a Courant algebroid isomorphism.
\end{proof}

\begin{lem}\label{Lem:CplusdPhiJproduceSameAsC}
Assume that $(C,\Connection)$ and $(C+\dee\Phi_{\JJJ },\Connection)$ are both \good\ pairs,
for some bundle map $\JJJ:F\to\GG$.
Then the Courant algebroids $E^s{(C,\Connection)}$ and $E^s{(C+\dee\Phi_{\JJJ},\Connection)}$ are isomorphic.
\end{lem}

\begin{proof}Let $E_1=E^s{(C,\Connection)}$, $E_2=E^s{(C+\dee\Phi_{\JJJ},\Connection)}$.
Clearly, $E_i$ is the standard Courant algebroid determined by the data $\ConnectionDer=\ConnectionDer^{\Connection}$,
$\Curvature=\Curvature^\Connection$ and $\HForm^i$, where
$\HForm^1=\HForm^{(C,\Connection)}$, $\HForm^2=\HForm^{(C+\dee\Phi_{\JJJ},\Connection)}$.
Since $(C,\Connection)$ and $(C+\dee\Phi_{\JJJ},\Connection)$ are both \good\ pairs,
we have $\inserts_{\gr}(\dee\Phi_{\JJJ})=0$, for all $\gr\in\GG$.
By Lemma~\ref{Lem:dPhiJJ}, one can check that this is equivalent to two conditions:
\[ \JJJ(F)\subset Z(\GG) \qquad \text{and} \qquad
\ConnectionDer_x\JJJ(y)-\ConnectionDer_y\JJJ(x)-\JJJ\lb{x}{y}=0,
\quad\forall x,y\in\sections{F} .\]
Using these, one has
\[ \HForm^2(x,y,z)-\HForm^1(x,y,z)=-\ipG{\JJJ(x)}{\Curvature(y,z)}+\CP, \]
for all $x,y,z\in \sections{F}$.
Then by Proposition~\ref{GeneralIsomorphism}, the map $I:E_1\to E_2$ defined by
\[ I(\xi+\gr+x)=(\xi+\beta(x)-2\varphis(\gr))+(\gr+\varphi(x))+x ,\]
where $\xi\in\Fs$; $\gr\in\GG$; $x\in F$; $\varphi=\thalf\JJJ$;
and $\beta=-\frac{1}{4}\JJJ^*\circ\JJJ$, is an isomorphism of Courant algebroids.
\end{proof}

The following Corollary implies that $\class{E^s{(\Abd;C,\Connection)}}$
actually does not depend on the choice of the hoist $\Connection$ of $\Abd$.

\begin{cor}\label{Lem:ChastwoConnections}
If $(C,\Connection)$ and $(C,\Connection')$ are both
coherent pairs on $\Abd$, then the Courant algebroids
$E^s{(C,\Connection)}$ and $E^s{(C,\Connection')}$
are isomorphic.
\end{cor}

\begin{proof} Suppose that
$\Connection'=\Connection+\DiffConnection$, for some bundle map
$\DiffConnection:F\to\GG$, then according to Lemma~\ref{Lem:goodpairStandardChange},
the pair $(C+\dee\Phi_{\DiffConnection},\Connection)$ is also \good.
By Lemma~\ref{Lem:almostidenticalgoodpairsIso} and~\ref{Lem:CplusdPhiJproduceSameAsC},
we have isomorphisms
$E^s{(C,\Connection')}\cong E^s{(C+\dee\Phi_{\DiffConnection},\Connection)}\cong E^s{(C,\Connection)}$.
\end{proof}

\begin{lem}\label{Lem:CplusanchorUpsdFomega}
If $(C,\Connection)$ is a \good\ pair, then so is the pair
$(C+\anchorUps(\dF\omega),\Connection)$, for any $2$-form
$\omega\in\sections{\wedge^2\Fs}$.
Moreover, the Courant algebroids $E^s{(C+\anchorUps(\dF\omega),\Connection)}$ and
$E^s{(C,\Connection)}$ are isomorphic.
\end{lem}

\begin{proof}
It is clear that $\inserts_{\gr}(\anchorUps(\dF\omega))=0$, $\forall\gr\in\GG$.
Thus one can easily check that $(C+\anchorUps(\dF\omega),\Connection)$ is a \good\ pair.
Let $E_i$ be the standard Courant algebroid determined by the quintuple
${(F,\GG;\ConnectionDer,\Curvature,\HForm^i)}$, where
$\HForm^1=\HForm^{(C,\Connection)}$,
$\HForm^2=\HForm^{(C+\anchorUps(\dF\omega),\Connection)}$.
Obviously, $\HForm^2-\HForm^1=\dF\omega$.
By Proposition~\ref{GeneralIsomorphism}, we can construct an
isomorphism $I:E_1\to E_2$ by setting
$I(\xi+\gr+x)=(\xi-\omega^\sharp(x))+\gr+x$,
for $\xi\in\Fs$, $\gr\in\GG$, and $x\in F$.
Here $\omega^\sharp:F\to\Fs$ is defined by
$\duality{\omega^\sharp(x)}{y}=\omega(x,y)$,
$\forall x,y\in\sections{F}$.
\end{proof}

\begin{lem}\label{Lem:C2C1homologic0producesame}
Let $(C_1,\Connection_1)$ and $(C_2,\Connection_2)$ be two
\good\ pairs. Assume that $C_2=C_1+\dee\varphi$, for some
$\varphi\in\Chain^2_{\leftrightarrow}(\Abd)$. Then the Courant
algebroids $E^s{(C_1,\Connection_1)}$ and $E^s{(C_2,\Connection_2)}$ are isomorphic.
\end{lem}

\begin{proof}
Choose $\DiffConnection:F\to\GG$ so that $\Connection_2=\Connection_1+\DiffConnection$.
Lemma~\ref{Lem:almostidenticalgoodpairsIso} implies that
\[ E^s{(C_2,\Connection_2)}=E^s{(C_2,\Connection_1+\DiffConnection)}
\cong E^s{(C_2+\dee \Phi_{\DiffConnection},\Connection_1)} =
E^s{(C_1+\dee (\varphi+\Phi_{\DiffConnection}),\Connection_1)} .\]
The hoist $\Connection_1$ produces a decomposition $\Abd\cong\GG\oplus F$.
Hence, there exist $\JJJ':F\to\GG$ and $\omega\in \sections{\wedge^2\Fs}$
such that
$\varphi+\Phi_{\DiffConnection}=\Phi_{\JJJ'}+\anchorUps\omega$.
By Lemmas~\ref{Lem:CplusanchorUpsdFomega} and~\ref{Lem:CplusdPhiJproduceSameAsC}, we have
\begin{equation*}
E^s{(C_1+\dee(\varphi+\Phi_{\DiffConnection}),\Connection_1)}
= E^s{(C_1+\dee\Phi_{\JJJ'}+\anchorUps(\dF\omega),\Connection_1)}
\cong E^s{(C_1+\dee \Phi_{\JJJ'},\Connection_1)}\cong E^s{(C_1,\Connection_1)}
. \qedhere \end{equation*}
\end{proof}

\begin{proof}[Proof of Proposition~\ref{Prop:Bijectivemapinversewelldefined}]
First we show that $\Bijectivemapinverse$ is well defined. Suppose
that two characteristic pairs $(\Abd, C)$ and $(\overline{\Abd},\overline{C })$
are equivalent, i.e.\ there is an isomorphism
$\sigma:\Abd\to\overline{\Abd}$ such that
$C-\sigma^*\overline{C}=\dee\varphi$,
for some $\varphi\in\Chain^2_{\leftrightarrow}(\Abd)$. We need to
show that there is an isomorphism of Courant algebroids
$E^s{(\Abd;C,\Connection)}\cong E^s{(\overline{\Abd};\overline{C},\overline{\Connection})}$.
It is quite obvious that
$(\sigma^*\overline{C},\sigma\inverse\circ\overline{\Connection})$
is also a \good\ pair on $\Abd$ and
\[ E^s{(\Abd;\sigma^*\overline{C},\sigma\inverse\circ\overline{\Connection})}
\cong E^s(\overline{\Abd};\overline{C},\overline{\Connection}) .\]
Since $C-\sigma^*\overline{C}=\dee\varphi$, Lemma~\ref{Lem:C2C1homologic0producesame}
gives the isomorphism
\[ E^s{(\Abd;C,\Connection)}\cong E^s{(\Abd;\sigma^*\overline{C},
\sigma\inverse\circ\overline{\Connection})} .\]
Thus we conclude that
$E^s{(\Abd;C,\Connection)}\cong E^s{(\overline{\Abd};\overline{C},\overline{\Connection})}$
and $\Bijectivemapinverse$ is well defined.
\end{proof}

With these preparations, we are able to prove the second main theorem in this paper.

\begin{proof}[Proof of Theorem~\ref{Thm:onetoone2}]
(1) Let us denote by $\ICRC$ the set of isomorphism classes of regular
Courant algebroids, and by $\ECCC$ the set of equivalence classes of
characteristic pairs. We need to establish a one-to-one
correspondence between $\ICRC$ and $\ECCC$.

For any regular Courant algebroid $E$, one may choose a dissection so as to identify
$E$ with $E^s=\Fs\oplus \GG\oplus F$. In Section~\ref{Sec:standard3form}, we
showed that there is a standard naive $3$-cocycle $C^s$ on $E^s$ and
it corresponds to a \good\ $3$-form on $\Abd^s=\GG\oplus F$, also
denoted by $C^s$.

So we can define a map $\Bijectivemap:\ICRC\to\ECCC$,
$\class{E}\mapsto\class{(\Abd^s,C^s)}$. By
Proposition~\ref{Lem:IPullBackDiffCoboundary}, this map is well
defined and does not depend on the choice of the dissection.

On the other hand, Proposition~\ref{Prop:Bijectivemapinversewelldefined}
indicates that there is a well defined map $\Bijectivemapinverse:\ECCC\to\ICRC$.
An easy verification shows that it is indeed the inverse of $\Bijectivemap$.
This completes the proof.

(2) This follows from Proposition~\ref{pro:3formclass}.
\end{proof}

\subsection{Classification}

\begin{lem}\label{Pro:allcoherentforms}
Let
$C\in\Gamma(\wedge^3\Abd^*)$ be  a coherent $3$-form.
Then any other coherent $3$-form is of the form
$C+\anchorUps(\varpi)+\dee\Phi_{\DiffConnection}$,
where $\varpi\in\sections{\wedge^3\Fs}$ is a closed $3$-form on $F$
and $\Phi_J$ is defined by~\eqref{Eqt:PhiJJ}.
\end{lem}

\begin{proof}
Assume that $C$ (resp. $C'$) is  $\hoist$ (resp. $\hoist'$)-coherent, for
 some hoist $\hoist$ (resp.  $\hoist'$) of $\Abd$.
The difference between $\hoist$ and $\hoist'$
determines a  bundle map $\DiffConnection:F\to\GG$
 such that
$\hoist'=\hoist-\DiffConnection$.
A direct calculation shows that the $3$-form
$C'-C-\dee\Phi_{\DiffConnection}$ must be the pull back of some
$3$-form $\varpi$ by the anchor  map $a$.
Conversely, it is easy to verify that $C+\anchorUps(\varpi)+\dee\Phi_{\DiffConnection}$
is coherent w.r.t.\ $\hoist'=\hoist-\DiffConnection$.
\end{proof}

\begin{defn}
A $3$-form $\vartheta\in \Chain^3_{\leftrightarrow}(\Abd)$
of a quadratic Lie algebroid $\Abd$ is said
to be \textbf{intrinsic}, if there exists a coherent $3$-form
$C$ and   an automorphism
$\sigma:~\Abd\rightarrow\Abd$ over the identity such that
$\vartheta=\sigma^*C-C$.
\end{defn}

It is simple to see that every intrinsic $3$-form is closed, and therefore
defines a class in  $H^3_{\leftrightarrow}(\Abd)$.
It turns out that this is the obstruction class
to lift the automorphism $\sigma$ to a Courant
algebroid automorphism.

\begin{lem}
\label{lem:I}
Let $\Abd_E$ be the ample Lie algebroid of a regular Courant algebroid
$E$ and $\vartheta=\sigma^*C-C$   an intrinsic form of $\Abd_E$. Then
$\class{\vartheta}=0$ if and only if $\sigma$ can be lifted to an
automorphism $\widetilde{\sigma}$ of the  Courant algebroid
$E$.
\end{lem}

Let  $\mathbb{I}$ denote the subset
in $H^3_{\leftrightarrow}(\Abd)$ consisting
of the cohomology classes  of all  intrinsic $3$-forms.
The following result is an easy
consequence of  Lemma \ref{Pro:allcoherentforms}.

\begin{prop}~~
\begin{itemize}
\item[1)] $\mathbb{I}$ consists of the cohomology classes
 $[\sigma^*C_0-C_0]$ for a fixed coherent $3$-form
$C_0$ and all automorphisms
$\sigma:~\Abd\rightarrow\Abd$ over the identity;
\item[2)]  $\mathbb{I}$ is an abelian subgroup of
$\anchorUps H^3(F)$.
\end{itemize}
\end{prop}

Hence the abelian group $\mathbb{I}$ measures how
many automorphisms of $\Abd$ that are essentially nontrivial,
 i.e., cannot   be lifted to the Courant algebroid level.

Let $\Phi$  be the natural map from
the isomorphism classes of regular Courant algebroids to
the isomorphism classes of quadratic Lie algebroids with vanishing
first Pontryagin class. We are now ready to prove the last main theorem
of this paper.

\begin{thm}\label{Thm:classification}
For any quadratic Lie algebroid $\ample$ with vanishing first
Pontryagin class,
$\Phi^{-1}([\ample])$ is isomorphic to
 $\frac{\anchorUps H^3(F)}{\mathbb{I}}\cong
\frac{H^3(F)}{(\anchorUps)\inverse(\mathbb{I})}$.
\end{thm}
\begin{proof}
By Theorem \ref{Thm:onetoone2}, we understand that every isomorphism
class of Courant algebroid in $\Phi^{-1}([\ample])$ can be
characterized by an equivalence class of
 characteristic pairs $ (\ample, C)$, where $C$ is
a coherent $3$-form.   Lemma \ref{Pro:allcoherentforms}
implies that $\anchorUps(H^3(F))$ acts transitively on $\Phi^{-1}([\ample])$.
Now we determine the isotropy group of this action.
Assume  that
$\varpi\in\sections{\wedge^3\Fs}$ is a closed $3$-form such that the
characteristic pairs $(\Abd,C)$ and $(\Abd,C+\anchorUps(\varpi))$
are equivalent.  By definition, there exists an automorphism
$\sigma$ of $\Abd$  such that
$[\sigma^*(C)-(C+\anchorUps(\varpi))]=0$ in $H^3_{\leftrightarrow}(\Abd)$.
 Hence it follows that $\class{\anchorUps(\varpi)}\in\mathbb{I}$.
Conversely, it is  simple  to see that $\mathbb{I}$ acts trivially on
$\Phi^{-1}([\ample])$.
This concludes the proof of the theorem.
\end{proof}

 $(\anchorUps)\inverse(\mathbb{I})$ is called  the intrinsic group of $\Abd$.
Note that the classification problem investigated here differs from
that studied by   \v{S}evera and Bressler  even in the transitive
case.
 \v{S}evera and Bressler classified the isomorphism classes of transitive Courant
algebroid extensions of a given Lie algebroid \cites{MR2360313, Severa},
while we are interested in isomorphism classes of Courant algebroids
themselves.

\begin{rmk}
{\em
Note that $a^*:H^3(F)\to H^3_{\leftrightarrow}(\Abd)$ is
not necessarily injective.

Let $M$ be a contact manifold with $\theta\in\Omega^1(M)$ being the contact
$1$-form.
Consider the transitive Lie algebroid $\Abd=(\RR\times M)\oplus TM$,
which is a central extension of Lie algebroid $TM$ by the exact $2$-form
$d\theta$ considered as a 2-cocycle of the Lie algebroid $TM$.
Equip the standard fiberwise metric on $\RR\times M\to M$.
It is clear that $\Abd$ is a quadratic Lie algebroid.
It is easy to check that $a^*[\theta d\theta]=0\in H^3_{\leftrightarrow}(\Abd)$.
However $[\theta d\theta]$ is a nontrivial class in $H^3_{DR}(M)$.

For instance, if $M=S^3$ with the canonical contact structure,
$H^3_{DR}(M)=\RR$ and $a^*H^3_{DR}(M)=0$. As a consequence, there is
a unique Courant algebroid (up to isomorphisms) whose induced
quadratic Lie algebroid is isomorphic to $\Abd=(\RR\times M)\oplus
TM$. On the other hand, according to~\cite{MR2360313,Severa},
Courant algebroid extensions of $\Abd=(\RR\times M)\oplus TM$ are
parametrized by $\RR$.}
\end{rmk}

In general, it is hard to describe explicitly the intrinsic group
$(\anchorUps)\inverse(\mathbb{I})$.
The following example, due to  \v{S}evera \cite{Severa}, gives
a nice
description of the intrinsic group in a special case.
We are grateful to \v{S}evera, who pointed  out this
example to us, which led us to find an error in an earlier version
of the draft.

\begin{ex}\label{Rmk:Severa}
Let  $\LieG$ be a compact  Lie algebra such that both
$\mathrm{Out}(\LieG)$  and $Z(\LieG)$ are  trivial. Assume that $M$
is  a connected and simply connected manifold. Choose an
ad-invariant non-degenerate bilinear form on  $\LieG$. Consider the
canonical quadratic Lie algebroid $\Abd= (M\times \LieG)\oplus TM$.
The Cartan $3$-form induces a  coherent $3$-form
 ${C}\in\sections{\wedge^3\Abd^*}$:
\begin{equation}\label{cartanform} {C}(\gr+x,\gs+y, \gt+z) =
-\ipLieG{\lbLieG{\gr}{\gs}}{\gt}  ,\end{equation}
$\forall \gr,\gs,\gt\in\cinf{M,\LieG}$ and $x,y,z\in\sections{TM}$.

\newcommand{\Adjoint}{\mathrm{Ad}}
\newcommand{\adjoint}{\mathrm{ad}}


Let $G$ be a connected and simply connected Lie group  whose Lie algebra is
$\LieG$. Then any automorphism $\sigma$ of $\Abd$ is  determined by
a smooth map $\tau:~M\to   G$, where $\sigma$ can be written as
$$
\sigma (\gr+v)=\Adjoint_{\tau (x)}(\gr-J(v))+v,\qquad~\forall~\gr\in
 { \LieG},  ~v\in T_xM.
$$
Here   $J:~ TM\rightarrow   \LieG$ is  the bundle map
defined  by:
\begin{equation}\label{Joutoftau}
J(v)=L_{\tau(x)^{-1}}(\tau_*(v)),\quad\forall ~v\in T_x M.
\end{equation}
To see this, one can consider $\Abd$ as the Atiyah Lie algebroid of
the trivial principal bundle $M\times G\to M$.
Then  $\sigma$ is induced by the principal bundle automorphism
defined by the gauge group element $\tau:~M\to   G$.

Using Lemma \ref{Lem:dPhiJJ}, one  obtains
$$
(\sigma^*C-C)(\gr+x,\gs+y, \gt+z) =-(\dee \Phi_J)(\gr+x,\gs+y,
\gt+z)+\ipLieG{\lbLieG{J(x)}{J(y)}}{J(z)}.
$$
It thus follows that
$$
\class{\sigma^*C-C}=-(J\circ\anchor)^*\class{\alpha}=-\anchorUps(J^*\class{\alpha}),
$$
where $\alpha =([\cdot , \cdot], \cdot) \in \wedge^3 \LieG^*$.
 It is
simple to see that $J^*({\alpha})=\tau^* (\tilde{\alpha})$,
where $\tilde{\alpha}$ is the corresponding Cartan $3$-form on $G$.
As a consequence,  the intrinsic group
$(\anchorUps)\inverse(\mathbb{I})$ is generated by elements of the
form $  \tau^*\class{\tilde{\alpha}}$, where $\tau~:M\rightarrow
G$ is any smooth map.
Note that $(\anchorUps)\inverse(\mathbb{I})$
 is discrete, because when $\tau_1$ and $\tau_2$ are
homotopic, then
$\tau_1^*\class{\tilde{\alpha}}=\tau_2^*\class{\tilde{\alpha}}$.

Here is  one  example. Let   $M=G=SU(2)$. Then $H^3_{DR}(M)\cong \RR$
and the intrinsic group $(\anchorUps)\inverse(\mathbb{I})$ is
isomorphic to $\ZZ$ with the generator being
 $\class{\tilde{\alpha}}$.
Hence by Theorem \ref{Thm:classification}, the isomorphism classes
of Courant algebroids with ample Lie algebroid isomorphic to the
quadratic Lie algebroid $(S^3 \times \frak{su}(2)
)\oplus T S^3 $  is  parametrized by $\RR/\mathbb{Z}\cong S^1$.
\end{ex}

%


\end{document}